\documentclass[11pt]{amsart}
\usepackage{amsmath,amssymb}
\newtheorem{theorem}{Theorem}

\newtheorem{proposition}{Proposition}

\newtheorem{lemma}{Lemma}
\theoremstyle{definition}

\newtheorem{question}{Question}
\newtheorem{rmk}{Remark}

\usepackage{latexsym,amssymb}
\begin{document}
\baselineskip=14.5pt
\title[Simultaneous indivisibility of class numbers]{Simultaneous indivisibility of class numbers of pairs of real quadratic fields} 

\author{Jaitra Chattopadhyay and Anupam Saikia}
\address[Jaitra Chattopadhyay]{Department of Mathematics, Indian Institute of Technology, Guwahati, Guwahati - 781039, Assam, India}
\email[Jaitra Chattopadhyay]{jaitra@iitg.ac.in, chat.jaitra@gmail.com}
\address[Anupam Saikia]{Department of Mathematics, Indian Institute of Technology, Guwahati, Guwahati - 781039, Assam, India}
\email[Anupam Saikia]{a.saikia@iitg.ac.in}

\begin{abstract}
For a square-free integer $t$, Byeon \cite{byeon} proved the existence of infinitely many pairs of quadratic fields $\mathbb{Q}(\sqrt{D})$ and $\mathbb{Q}(\sqrt{tD})$ with $D > 0$ such that the class numbers of all of them are indivisible by $3$. In the same spirit, we prove that for a given integer $t \geq 1$ with $t \equiv 0 \pmod {4}$, a positive proportion of fundamental discriminants $D > 0$ exist for which the class numbers of both the real quadratic fields $\mathbb{Q}(\sqrt{D})$ and $\mathbb{Q}(\sqrt{D + t})$ are indivisible by $3$. This also addresses the complement of a weak form of a conjecture of Iizuka in \cite{iizuka}. As an application of our main result, we obtain that for any integer $t \geq 1$ with $t \equiv 0 \pmod{12}$, there are infinitely many pairs of real quadratic fields $\mathbb{Q}(\sqrt{D})$ and $\mathbb{Q}(\sqrt{D + t})$ such that the Iwasawa $\lambda$-invariants associated with the basic $\mathbb{Z}_{3}$-extensions of both $\mathbb{Q}(\sqrt{D})$ and $\mathbb{Q}(\sqrt{D + t})$ are $0$. For $p = 3$, this supports Greenberg's conjecture which asserts that $\lambda_{p}(K) = 0$ for any prime number $p$ and any totally real number field $K$.
\end{abstract}

\renewcommand{\thefootnote}{}



\renewcommand{\thefootnote}{\arabic{footnote}}
\setcounter{footnote}{0}

\maketitle

\section{Introduction}

For a square-free integer $D > 0$, we denote the class number of the real quadratic field $\mathbb{Q}(\sqrt{D})$ by $h(D)$ and that of the imaginary quadratic field $\mathbb{Q}(\sqrt{-D})$ by $h(-D)$. It is well-known that $h(-D) \to \infty$ as $D \to \infty$. However, it is unknown whether $h(D)$ becomes unbounded as $D \to \infty$. In fact, the famous {\it Gauss class number $1$ conjecture} asserts that $h(D) = 1$ holds infinitely often. This motivated number theorists to study the divisibility and indivisibility properties of $h(D)$, apart from studying the more sophisticated and intricate problem of the classification of quadratic fields having a given class number.

\smallskip

Towards the direction of the question of the divisibility of class numbers of quadratic fields, Nagell \cite{nagell} and later Ankeny and Chowla \cite{AC} proved that for any given integer $g \geq 2$, there exist infinitely many imaginary quadratic fields whose class numbers are all divisible by $g$. Later, Weinberger \cite{berger} proved its analogue for real quadratic fields. The quantitative studies have also been  done for some particular values of $g$ (cf. \cite{kalyan}, \cite{self-jrms}, \cite{luca}, \cite{sound} etc.). Besides these, the study of the divisibility properties of class numbers turns out to be quite interesting when we work with more than one quadratic field. In recent years, substantial progress has been made for pairs of quadratic fields (cf. \cite{iizuka}, \cite{komatsu-acta}, \cite{komatsu-ijnt}). In fact, there is a very interesting conjecture by Iizuka \cite{iizuka} that asserts that for any prime $\ell \geq 3$ and any integer $m \geq 1$, there exist infinitely many real or imaginary quadratic fields $\mathbb{Q}(\sqrt{D}), \mathbb{Q}(\sqrt{D + 1}),\ldots ,\mathbb{Q}(\sqrt{D + m})$ with $D \in \mathbb{Z}$ such that the class numbers of all of them are divisible by $\ell$. Very recently, a weak form of this conjecture is addressed in \cite{self-acta} for triples of imaginary quadratic fields.

\smallskip

Like the divisibility problem, it is quite useful to know the distribution of the real quadratic fields $\mathbb{Q}(\sqrt{D})$ for which $h(D)$ is not divisible by a given integer $g$. In view of this, the question of indivisibility of class numbers of quadratic fields draws attention of mathematicians as this is closely related to the Gauss class number $1$ conjecture (cf. \cite{olivia}, \cite{byeon}, \cite{davenport}, \cite{kohnen-ono}, \cite{naka-horie}, \cite{ono}). As in the problem of the simultaneous divisibility of class numbers, Byeon \cite{byeon} addressed the question of simultaneous $3$-indivisibility of class numbers of quadratic fields and proved the following.

\begin{proposition} \cite{byeon}\label{byeon}
Let $t$ be a square-free integer. Then there exist infinitely many fundamental discriminants $D > 0$ with a positive density such that the class numbers of both the quadratic fields $\mathbb{Q}(\sqrt{D})$ and $\mathbb{Q}(\sqrt{tD})$ are indivisible by $3$.
\end{proposition}

Motivated by Proposition \ref{byeon} and the conjecture of Iizuka, we ask the following question.

\begin{question}\label{quest}
Let $\ell \geq 3$ be a prime number and let $t$ be an integer. Do there exist infinitely many pairs of real or imaginary quadratic fields of the form $\mathbb{Q}(\sqrt{D})$ and $\mathbb{Q}(\sqrt{D+t})$ such that the class numbers of all the fields are indivisible by $\ell$?
\end{question}

In this paper, following the method used by Byeon in \cite{byeon}, we give an affirmative answer to Question \ref{quest} for pairs of real quadratic fields and for $\ell = 3$. More precisely, we prove the following theorem.

\begin{theorem}\label{main-th}
Let $t \geq 1$ be an integer with $t \equiv 0 \pmod {4}$. Then there exist infinitely many fundamental discriminants $D > 0$ with positive density such that the class numbers of the real quadratic fields $\mathbb{Q}(\sqrt{D})$ and $\mathbb{Q}(\sqrt{D+t})$ are all indivisible by $3$.
\end{theorem}

\smallskip

\begin{rmk}
In Section \ref{section-4}, we give an application of Theorem \ref{main-th} to the vanishing of Iwasawa $\lambda$-invariants for pairs of real quadratic fields. In particular, we prove that for an integer $t \geq 1$ with $t \equiv 0 \pmod {12}$, there exist infinitely many pairs of real quadratic fields $\mathbb{Q}(\sqrt{D})$ and $\mathbb{Q}(\sqrt{D + t})$ for which the Iwasawa $\lambda$-invariant vanishes for the rational prime $3$.
\end{rmk}

\section{Preliminaries}

First, we recall that an integer $D$ is a {\it fundamental discriminant} if it is the discriminant of a quadratic field. That is, either $D$ is square-free with $D \equiv 1 \pmod {4}$ or $D = 4m$ for some square-free integer $m \equiv 2,3 \pmod {4}$. We record the following result by Nakagawa and Horie \cite{naka-horie} that has been crucially used by Byeon in \cite{byeon} to prove Proposition \ref{byeon}.

\begin{proposition} \cite{naka-horie}\label{1st-prop}
Let $m \geq 1$ and $N \geq 1$ be two integers satisfying the following two conditions.
\begin{enumerate}
\item If $p$ is an odd prime number such that $p \mid \gcd(m,N)$, then $N \equiv 0 \pmod {p^{2}}$ and $m \not\equiv 0 \pmod{p^{2}}$.
\item If $N$ is an even integer, then either $N \equiv 0 \pmod {4}$ and $m \equiv 1 \pmod {4}$ or $N \equiv 0 \pmod {16}$ and $m \equiv 8, 12 \pmod {16}$.
\end{enumerate} 
For a positive real number $X$, let $S_{+}(X)$ stand for the set of positive fundamental discriminants $D < X$ and let $$S_{+}(X,m,N) = \{D \in S_{+}(X) : D \equiv m \pmod {N}\}.$$
For a fundamental discriminant $D > 0$, let $r_{3}(D)$ be the $3$-rank of the ideal class group $Cl_{\mathbb{Q}(\sqrt{D})}$ of $\mathbb{Q}(\sqrt{D})$. Then we have 
\begin{equation}\label{eq-1}
\displaystyle\lim_{X \to \infty} \frac{\displaystyle\sum_{D \in S_{+}(X,m,N)} 3^{r_{3}(D)}}{|S_{+}(X,m,N)|} = \frac{4}{3}.
\end{equation}
\end{proposition}

\medskip

The next lemma is directly obtained from Proposition \ref{1st-prop} and is used in \cite{byeon} without proof. Here we furnish a detailed proof of this for the sake of completeness.

\begin{lemma} \cite{byeon}\label{byeon-lemma}
Let $m$ and $N$ be two positive integers satisfying the hypotheses of Proposition \ref{1st-prop} and let $D$ be a fundamental discriminant. Then 
\begin{equation}
\displaystyle\liminf_{X \to \infty} \frac{|\{D \in S_{+}(X,m,N) : h(D) \not\equiv 0 \pmod {3}\}|}{|S_{+}(X,m,N)|} \geq \frac{5}{6}.
\end{equation}
\end{lemma}
\begin{proof}
For $D \in S_{+}(X,m,N)$, let $r_{3}(D)$ be the $3$-rank of $Cl_{\mathbb{Q}(\sqrt{D})}$. We consider the sum $\displaystyle\sum_{D \in S_{+}(X,m,N)} 3^{r_{3}(D)}$ and collect the like powers of $3$ together to obtain the following.
\begin{eqnarray*}
\displaystyle\sum_{D \in S_{+}(X,m,N)} 3^{r_{3}(D)} &=& \displaystyle\sum_{\substack{{D \in S_{+}(X,m,N)} \\ r_{3}(D) = 0}} 3^{r_{3}(D)} + \displaystyle\sum_{\substack{{D \in S_{+}(X,m,N)} \\ r_{3}(D) \geq 1 }} 3^{r_{3}(D)} \\ &\geq & \displaystyle\sum_{\substack{{D \in S_{+}(X,m,N)} \\ r_{3}(D) = 0}} 3^{r_{3}(D)} + \displaystyle\sum_{\substack{{D \in S_{+}(X,m,N)} \\ r_{3}(D) \geq 1 }} 3 \\ &=& \displaystyle\sum_{\substack{{D \in S_{+}(X,m,N)} \\ r_{3}(D) = 0}} 3^{r_{3}(D)} + 3\left(\displaystyle\sum_{D \in S_{+}(X,m,N)} 1 - \displaystyle\sum_{\substack{{D \in S_{+}(X,m,N)} \\ r_{3}(D) = 0 }} 1 \right) \\ &=& 3\left(\displaystyle\sum_{D \in S_{+}(X,m,N)} 1\right)- 2\left(\displaystyle\sum_{\substack{{D \in S_{+}(X,m,N)} \\ r_{3}(D) = 0 }} 1\right).
\end{eqnarray*}

Therefore, we get $$\frac{2\left(\displaystyle\sum_{\substack{{D \in S_{+}(X,m,N)} \\ r_{3}(D) = 0 }} 1\right)}{\displaystyle\sum_{D \in S_{+}(X,m,N)} 1} \geq 3 - \frac{\displaystyle\sum_{D \in S_{+}(X,m,N)} 3^{r_{3}(D)}}{\displaystyle\sum_{D \in S_{+}(X,m,N)} 1}.$$

Consequently, from equation \eqref{eq-1}, we get $$2\liminf_{X \to \infty} \frac{\left(\displaystyle\sum_{\substack{{D \in S_{+}(X,m,N)} \\ r_{3}(D) = 0 }} 1\right)}{\displaystyle\sum_{D \in S_{+}(X,m,N)} 1} \geq 3 - \frac{4}{3}.$$


Hence we have 
\begin{equation*}
\displaystyle\liminf_{X \to \infty} \frac{|\{D \in S_{+}(X,m,N) : h(D) \not\equiv 0 \pmod {3}\}|}{|S_{+}(X,m,N)|} \geq \frac{5}{6}.
\end{equation*}
\end{proof}

\medskip

We shall make use of an equivalent formulation of an integer being square-free in terms of the M\"{o}bius function $\mu$. In other words, an integer $a$ is square-free if and only if $\mu(a) \neq 0$. The next lemma gives us an asymptotic formula for the number of square-free integers in an arithmetic progression and is taken from \cite{sq-free}. 

\begin{lemma} \cite{sq-free}\label{sq-free}
Let $k \geq 1$ and $\ell \geq 1$ be two integers with $\gcd(k,\ell) = 1$. For a large positive real number $X$, let $$\mathcal{Q}(X,k,\ell) = |\{m \in \mathbb{N} : m \leq X, m \equiv \ell \pmod {k} \mbox{ and } \mu(m) \neq 0\}|.$$ Then for any real number $\varepsilon > 0$, we have 
\begin{equation}
\mathcal{Q}(X,k,\ell) = \frac{6}{k\pi^{2}}\displaystyle\prod_{p \mid k}\left(1 - \frac{1}{p^{2}}\right)^{-1}X + O(X^{\frac{1}{2}}k^{-\frac{1}{4} + \varepsilon} + k^{\frac{1}{2} + \varepsilon}),
\end{equation}
and the error term is uniform in $k$.
\end{lemma}

\section{Proof of Theorem \ref{main-th}}

Let $t \geq 1$ be a given integer with $t \equiv 0 \pmod {4}$. We choose two positive integers $m$ and $N$ in such a way that $\gcd(m,N) = \gcd(m + t, N) = 1$, $m \equiv 1 \pmod {4}$ and $N \equiv 0 \pmod {4}$. Thus the hypotheses on $m$, $N$, $m + t$ and $N$ in Proposition \ref{1st-prop} are satisfied. For this choice of $m$ and $N$ and for a large positive real number $X$, let $$L(X) = \{D \leq X : D \equiv m \pmod {N}, \mu(D) \neq 0 \mbox{ and } h(D) \not\equiv 0 \pmod {3}\} $$ and $$L_{t}(X) = \{D \leq X : D \equiv m \pmod {N}, \mu(D + t) \neq 0 \mbox{ and } h(D + t) \not\equiv 0 \pmod {3}\}.$$

Our assumptions on $m, t, N$ and the properties of $L(X)$ and $L_{t}(X)$ imply that $D \equiv m \equiv 1 \pmod {4}$ and $D + t \equiv m + t \equiv 1 \pmod {4}$. Therefore, the integers in the sets $L(X)$ and $L_{t}(X)$ are indeed fundamental discriminants. Also, we immediately have $$L(X) \cap L_{t}(X) = \{D \leq X : D \equiv m \pmod {N}, \mu(D)\mu(D + t) \neq 0 \mbox{ and } h(D)h(D + t) \not\equiv 0 \pmod {3}\}.$$

Let $\mathcal{S}(X) = \{D \leq X : D \equiv m \pmod {N}\}$. Now, since $\gcd(m,N) = 1$, by Lemma \ref{sq-free}, we have 
\begin{equation}\label{for-l-1}
\displaystyle\lim_{X \to \infty} \frac{\mathcal{Q}(X,N,m)}{|\mathcal{S}(X)|} = \displaystyle\frac{6}{\pi^{2}}\displaystyle\prod_{p \mid N}\left(1 - \frac{1}{p^{2}}\right)^{-1} \geq \frac{6}{\pi^{2}}.
\end{equation}

Therefore, from equation \eqref{for-l-1} and Lemma \ref{byeon-lemma}, we get
\begin{equation}\label{up-1}
\displaystyle\liminf_{X \to \infty} \frac{|L(X)|}{|\mathcal{S}(X)|} \geq \frac{5}{6}\cdot \frac{6}{\pi^{2}} = \frac{5}{\pi^{2}}.
\end{equation}

Similarly, since $\gcd(m + t, N) = 1$ and $D \equiv m \pmod {N}$ is equivalent to $D + t \equiv m + t \pmod {N}$, by Lemma \ref{sq-free}, we have 
\begin{equation}\label{for-l-2}
\displaystyle\lim_{X \to \infty} \frac{\mathcal{Q}(X,N,m + t)}{|\mathcal{S}(X)|} = \displaystyle\frac{6}{\pi^{2}}\displaystyle\prod_{p \mid N}\left(1 - \frac{1}{p^{2}}\right)^{-1} \geq \frac{6}{\pi^{2}}.
\end{equation}

Hence from equation \eqref{for-l-2} and Lemma \ref{byeon-lemma}, we get 
\begin{equation}\label{up-2}
\displaystyle\liminf_{X \to \infty} \frac{|L_{t}(X)|}{|\mathcal{S}(X)|} \geq \frac{5}{6}\cdot \frac{6}{\pi^{2}} = \frac{5}{\pi^{2}}.
\end{equation}

Therefore, using the principle of inclusion-exclusion, we get 
\begin{equation}\label{new-eqn}
\frac{|L(X)\cap L_{t}(X)|}{|\mathcal{S}(X)|} = \frac{|L(X)|}{|\mathcal{S}(X)|} +  \frac{|L_{t}(X)|}{|\mathcal{S}(X)|} -  \frac{|L(X)\cup L_{t}(X)|}{|\mathcal{S}(X)|}.
\end{equation}

Since $\displaystyle\frac{|L(X)\cup L_{t}(X)|}{|\mathcal{S}(X)|}$ can be at most $1$, from equation \eqref{up-1}, \eqref{up-2} and \eqref{new-eqn}, we obtain
\begin{eqnarray*}
\displaystyle\liminf_{X \to \infty} \frac{|L(X)\cap L_{t}(X)|}{|\mathcal{S}(X)|} &\geq & \frac{5}{\pi^{2}} + \frac{5}{\pi^{2}} - 1
\\ &=& \frac{10 - \pi^{2}}{\pi^{2}} > 0.
\end{eqnarray*}

Therefore, There exist an infinite family of pairs of real quadratic fields of the form $\mathbb{Q}(\sqrt{D})$ and $\mathbb{Q}(\sqrt{D + t})$ such that the class numbers of all of them are indivisible by $3$. This completes the proof of Theorem \ref{main-th}. $\hfill\Box$

\section{Application to Iwasawa $\lambda$-invariants}\label{section-4}

For a quadratic field $k$ and a prime number $p$, it is a well-known result due to Iwasawa \cite{iwasawa} that if $p$ does not split completely in $k$ and $p$ does not divide the class number $h_{k}$ of $k$, then $\lambda_{p}(k) = 0$, where $\lambda_{p}(k)$ is the Iwasawa $\lambda$-invariant associated with the basic $\mathbb{Z}_{p}$-extension over $k$ (also, cf. \cite{naka-horie}).

\smallskip

In this section, we deal with the case when $k$ is a real quadratic field and $p = 3$. Let $t \geq 1$ be an integer with $t \equiv 0 \pmod {12}$. We choose integers $m \geq 1$ and $N \geq 1$ in such a way that $\gcd(m,N) = \gcd(m + t,N) = 1$, $m \equiv 5 \pmod{12}$ and $N \equiv 0 \pmod {12}$. Then the hypotheses of Proposition \ref{1st-prop} on the integers $m, N \mbox{ and } N$ are satisfied. Now, as in the proof of Theorem \ref{main-th}, for $D \in L(X)$, we have $D \equiv m \pmod{N}$. Since $m \equiv 5 \pmod{12}$ and $N \equiv 0 \pmod {12}$, we have $D \equiv m \equiv 2 \pmod {3}$. Similarly, since $t \equiv 0 \pmod {12}$, for $D \in L_{t}(X)$ we have $D + t \equiv m + t \equiv 2 \pmod{3}$. Consequently, each of the Legendre symbols $\left(\frac{D}{3}\right)$ and $\left(\frac{D + t}{3}\right)$ is equal to $\left(\frac{2}{3}\right) = -1$. Thus for $D \in L(X) \cap L_{t}(X)$, the rational prime $3$ does not split completely in $\mathbb{Q}(\sqrt{D})$ and $\mathbb{Q}(\sqrt{D + t})$ and their class numbers are also indivisible by $3$. Therefore, $\lambda_{3}(\mathbb{Q}(\sqrt{D})) = \lambda_{3}(\mathbb{Q}(\sqrt{D + t})) = 0$. This provides us an infinite family of pairs of real quadratic fields for which the Iwasawa $\lambda$-invariant $\lambda_{3}$ are all $0$.

\smallskip

\begin{rmk}
In \cite{taya}, Taya proved the existence of infinitely many real quadratic fields $K = \mathbb{Q}(\sqrt{D})$ in which the rational prime $3$ splits completely and $\lambda_{3}(K) = 0$. That is, there exist  infinitely many real quadratic fields $K = \mathbb{Q}(\sqrt{D})$ with $D \equiv 1 \pmod {3}$ and $\lambda_{3}(K) = 0$. Our result addresses the complement of Taya's result and deals with the case of simultaneous vanishing of $\lambda_{3}$ in $\mathbb{Q}(\sqrt{D})$ and $\mathbb{Q}(\sqrt{D + t})$ where the rational prime $3$ stays inert in both the fields.
\end{rmk}

\smallskip

\begin{rmk}
It is a famous conjecture by Greenberg (cf. \cite{taya}) which asserts that for any prime number $p$ and any totally real number field $K$, we have $\lambda_{p}(K) = 0$. Our result is in support of this conjecture for the particular case $p = 3$.
\end{rmk}

\section{Concluding Remarks}

The method used in the proof of Theorem \ref{main-th} does not yield the analogous result for imaginary quadratic fields. To see this, let $m$ and $N$ be two positive integers satisfying the hypotheses of Proposition \ref{1st-prop} and for a large positive real number $X$, let $S_{-}(X,m,N)$ be the set of all fundamental discriminants $D$ with $-X < D < 0$ and $D \equiv m \pmod {N}$. Then from Proposition \ref{1st-prop}, it follows that (cf. \cite{byeon}) 
\begin{equation*}
\displaystyle\liminf_{X \to \infty} \frac{|\{D \in S_{-}(X,m,N) : h(-D) \not\equiv 0 \pmod {3}\}|}{|S_{-}(X,m,N)|} \geq \frac{1}{2}.
\end{equation*}
If we construct the sets $L^{\prime}(X)$ and $L^{\prime}_{t}(X)$ analogously for the case of imaginary quadratic fields, then a non-trivial intersection of $L^{\prime}(X)$ and $L^{\prime}_{t}(X)$ is not guaranteed. Consequently, in this case, we are unable to conclude anything about the infinitude of some such pairs of imaginary quadratic fields with class numbers indivisible by $3$.

\medskip

{\bf Acknowledgements.} It is a pleasure for the first author to thank Indian Institute of Technology, Guwahati for the financial support.


\begin{thebibliography}{9999}

\bibitem{AC}
N. Ankeny and S. Chowla, {\it On the divisibility of the class numbers of quadratic fields}, {\sf Pacific J. Math.}, {\bf 5} (1955), 321-324.

\bibitem{olivia}
O. Beckwith, {\it Indivisibility of class numbers of imaginary quadratic fields}, {\sf Res. Math. Sci.}, {\bf 4} (2017), Paper No. 20, 11 pp.

\bibitem{byeon}
D. Byeon, {\it Class numbers of quadratic fields $\mathbb{Q}(\sqrt{D})$ and $\mathbb{Q}(\sqrt{tD})$}, {\sf Proc. Amer. Math. Soc.}, {\bf 132} (2004), 3137-3140.

\bibitem{kalyan}
K. Chakraborty and M. Ram Murty, {\it On the number of real quadratic fields with class number divisible by $3$}, {\sf Proc. Amer. Math. Soc.}, {\bf 131} (2002), 41-44.

\bibitem{self-jrms}
J. Chattopadhyay, {\it A short note on the divisibility of class number of real quadratic fields}, {\sf J. Ramanujan Math. Soc.}, {\bf 34} (2019), 389-392.

\bibitem{self-acta}
J. Chattopadhyay and S. Muthukrishnan, {\it On the simultaneous $3$-divisibility of class numbers of triples of imaginary quadratic fields}, {\sf Acta Arith.}, DOI: 10.4064/aa200221-16-6

\bibitem{davenport}
H. Davenort and H. Heilbronn, {\it On the density of discriminants of cubic fields}, {\sf Proc. Royal Soc. A}, {\bf 322} (1971) 405-420.

\bibitem{iizuka}
Y. Iizuka, {\it On the class number divisibility of pairs of imaginary quadratic fields}, {\sf J. Number Theory}, {\bf 184} (2018), 122-127.

\bibitem{iwasawa}
K. Iwasawa, {\it A note on class numbers of algebraic number fields}, {\sf Abh. Math. Sem. Univ. Hamburg}, {\bf 20} (1956), 257-258.

\bibitem{kishi-miyake}
Y. Kishi and K. Miyake, {\it Parametrization of the quadratic fields whose class numbers are divisible by three}, {\sf J. Number Theory}, {\bf 80} (2000), 209-217.

\bibitem{kohnen-ono}
W. Kohnen, K. Ono, {\it Indivisibility of class numbers of imaginary quadratic fields and orders of Tate-Shafarevich groups of elliptic curves with complex multiplication}, {\sf Invent. Math.}, {\bf 135} (1999), 387-398.

\bibitem{komatsu-acta}
T. Komatsu, {\it An infinite family of pairs of quadratic fields $\mathbb{Q}(\sqrt{D})$ and $\mathbb{Q}(\sqrt{mD})$ whose class numbers are both divisible by $3$}, {\sf Acta Arith.}, {\bf 104} (2002), 129-136.

\bibitem{komatsu-ijnt}
T. Komatsu, {\it An infinite family of pairs of imaginary quadratic fields with ideal classes of a given order}, {\sf Int. J. Number Theory}, {\bf 13} (2017), 253-260. 

\bibitem{luca}
F. Luca, {\it A note on the divisibility of class numbers of real quadratic fields}, {\sf C. R. Math. Acad. Sci. Soc. R. Can}, {\bf 25} (2003), 71-75.

\bibitem{nagell}
T. Nagell, {\it Uber die Klassenzahl imaginar quadratischer Zahkorper}, {\sf Abh. Math. Seminar Univ. Hamburg}, {\bf 1} (1922), 140-150.

\bibitem{naka-horie}
J. Nakagawa and K. Horie, {\it Elliptic curves with no torsion points}, {\sf Proc. Amer. Math. Soc.}, {\bf 104} (1988), 20-25.

\bibitem{ono}
K. Ono, {\it Indivisibility of class numbers of real quadratic fields}, {\sf Compositio Math.}, {\bf 119} (1999), 1-11.

\bibitem{sq-free}
K. Prachar, {\it Über die kleinste quadratfreie Zahl einer arithmetischen Reihe}, {\sf Monatsh. Math.}, {\bf 62} (1958), 173-176.

\bibitem{scholz}
A. Scholz, {\it Über die Beziehung der Klassenzahlen quadratischer Körper zueinander}, {\sf J. Reine Angew. Math.}, {\bf 166} (1932), 201-203.

\bibitem{sound}
K. Soundararajan, {\it Divisibility of class numbers of imaginary quadratic fields}, {\sf J. London Math. Soc.}, {\bf 61} (2000), 681-690.

\bibitem{taya}
H. Taya, {\it Iwasawa invariants and class numbers of quadratic fields for the prime $3$}, {\sf Proc. Amer. Math. Soc.}, {\bf 128} (2000), 1285-1292.

\bibitem{berger}
P. Weinberger, {\it Real quadratic fields with class numbers divisible by $n$}, {\sf J. Number Theory}, {\bf 5} (1973), 237-241.









\end{thebibliography}
\end{document}